\documentclass[11pt,reqno]{amsart}
\usepackage{amsthm, amsfonts, amssymb, color}
 \usepackage{mathrsfs}
\usepackage{enumerate}
\usepackage{amsmath}
 \usepackage{amstext, amsxtra}
  \usepackage{txfonts}
 \usepackage[colorlinks, linkcolor=black, citecolor=blue, pagebackref, hypertexnames=false]{hyperref}

 \allowdisplaybreaks
\newtheorem{theorem}{Theorem}[section]
\newtheorem{proposition}[theorem]{Proposition}

\newtheorem{lemma}[theorem]{Lemma}

\newtheorem{remark}[theorem]{Remark}

\newcommand\R{\mathbb{R}}

\newcommand{\la}{\lambda}

\newcommand{\supp}{{\rm supp}{\hspace{.05cm}}}

\numberwithin{equation}{section}
\theoremstyle{definition}

\title
[fractional Schr\"{o}dinger-Poisson systems]
 {The existence and concentration of positive ground state solutions for a class of fractional Schr\"{o}dinger-Poisson systems with steep potential well}
\author{Liejun Shen \ and \ Xiaohua Yao}

\address{ Liejun Shen,  Hubei Key Laboratory of Mathematical Sciences and School of Mathematics and Statistics,
 Central China Normal University, Wuhan, 430079, P. R. China }
\email{liejunshen@sina.com}
\address{Xiaohua Yao, Hubei Key Laboratory of Mathematical Sciences and School of Mathematics and Statistics,
 Central China Normal University, Wuhan, 430079, P.R. China}
\email{yaoxiaohua@mail.ccnu.edu.cn }
\date{\today}

\begin{document}

\maketitle
\begin{abstract} The present study is concerned with the following fractional Schr\"{o}dinger-Poisson system with steep potential well:
$$
  \left\{%
\begin{array}{ll}
    (-\Delta)^s u+ \la V(x)u+K(x)\phi u= f(u), & x\in\R^3, \\
    (-\Delta)^t \phi=K(x)u^2, &  x\in\R^3,\\
\end{array}%
\right.
$$
where $s,t\in(0,1)$ with $4s+2t>3$, and $\la>0$ is a parameter. Under certain assumptions on $V(x)$, $K(x)$ and $f(u)$ behaving like $|u|^{q-2}u$ with $2<q<2_s^*=\frac{6}{3-2s}$, the existence of positive ground state solutions and concentration results are obtained via some new analytical skills and Nehair-Poho\v{z}aev identity. In particular, the monotonicity assumption on the nonlinearity is not necessary.
\end{abstract}


\section{Introduction and main results}
In the present paper, we are concerned with the existence and concentration of positive ground state solutions for the following fractional Schr\"{o}dinger-Poisson system:
\begin{equation}\label{mainequation1}
    \left\{%
\begin{array}{ll}
    (-\Delta)^s u+ \la V(x)u+K(x)\phi u= f(u), & x\in\R^3, \\
    (-\Delta)^t \phi=K(x)u^2, &  x\in\R^3,\\
\end{array}%
\right.
\end{equation}
where $s,t\in(0,1)$, $4s+2t>3$ and the parameter $\la>0$. On the potential $V(x)$, we need to make the following assumptions:
\vskip0.3cm
\begin{enumerate}[$(V_1)$]
  \item \emph{$V(x)\in C(\R^3,\R)$ with $V(x)\geq 0$ on $\R^3$;}
  \end{enumerate}
  \vskip0.3cm
\begin{enumerate}[$(V_2)$]
  \item \emph{there exists $c>0$ such that the set $\{V< c\}\triangleq\big\{x\in\R^3: V(x) < c\big\}$ has positive finite Lebesgue measure;}
      \end{enumerate}
       \vskip0.3cm
\begin{enumerate}[$(V_3)$]
\item  \emph{$\Omega=\text{int}V^{-1}(0)$
is nonempty and has smooth boundary with $\overline{\Omega}=V^{-1}(0)$, where $V^{-1}(0)\triangleq\{x\in\R^3: V (x) = 0\}$}.
\end{enumerate}
\vskip0.3cm
In their celebrated paper, T. Bartsch and Z. Wang \cite{Bartsch0} firstly proposed the above hypotheses to study a nonlinear Schr\"{o}dinger equation. The potential $\la V(x)$ with assumptions $(V_1)-(V_3)$ usually are called by the steep potential well.

Let us recall the history of the study for Schr\"{o}dinger-Poisson system
\begin{equation}\label{introduction2}
 \left\{%
\begin{array}{ll}
    -\Delta u+V(x)u+\phi u=f(x,u), & x\in\R^3, \\
    -\Delta \phi=u^2, &  x\in\R^3.\\
\end{array}%
\right.
\end{equation}
Due to the real physical meaning, the system \eqref{introduction2}
has been studied extensively by many scholars in the last several decades. Benci and Fortunato \cite{Benci1} introduced the system like \eqref{introduction2} to
describe solitary waves for nonlinear Schr\"{o}dinger type equations and look for the existence of standing waves interacting with an unknown electrostatic field. We refer the readers to \cite{Benci1,Benci2} and the references therein to get a more physical background of the system \eqref{introduction2}. Nearly Y. Jiang and H. Zhou \cite{Jiang} firstly applied the steep potential well to the Schr\"{o}dinger-Poisson system and proved the existence of nontrivial solutions and ground state solutions. Subsequently
by using the linking theorem \cite{Rabinowitz,Willem}, L. Zhao, H. Liu and F. Zhao \cite{Zhao} studied the existence and concentration of nontrivial solutions for the following Schr\"{o}dinger-Poisson system
\begin{equation}\label{Zhao}
    \left\{%
\begin{array}{ll}
    -\Delta u+ \la V(x)u+K(x)\phi u= |u|^{p-2}u, & x\in\R^3, \\
    -\Delta  \phi=K(x)u^2, &  x\in\R^3,\\
\end{array}%
\right.
\end{equation}
under the conditions
\begin{enumerate}[$\widetilde{({V_1})}$]
\item \emph{$V(x)\in C(\R^3,\R)$ and $V$ is bounded form below;}
\end{enumerate}
and $(V_2)-(V_3)$ with some suitable assumptions on $K(x)$ for $4\leq p<6$. It is worth mentioning that they specially established the existence and concentration of nontrivial solutions to \eqref{Zhao} by L. Jeanjean's monotonicity trick \cite{Jeanjean2} under the conditions $(V_1)-(V_3)$, $K(x)\geq 0$ for $x\in\R^3$ with $K(x)\in L^{\infty}_{\text{loc}}(\R^3)\cap L^2(\R^3)$ and
\begin{enumerate}[$\widetilde{({V_4})}$]
  \item \emph{$V(x)$ is weakly differentiable such that $(x,\nabla V)\in L^{p_1}(\R^3)$ for some $p_1\in[\frac{3}{2},\infty]$, and
 \[
2V(x)+(x,\nabla V)\geq 0,\ \ \text{for}\ \ a.e\ \ x\in\R^3,
\]
where $(\cdot,\cdot)$ is the usual inner product in $\R^3$.}
\end{enumerate}
\vskip3mm
\begin{enumerate}[$\widetilde{({K})}$]
  \item \emph{$K(x)$ is weakly differentiable such that $(x,\nabla K)\in L^{p_2}(\R^3)$ for some $p_2\in[2,\infty]$, and
\[
\frac{2(p-3)}{p}K(x)+(x,\nabla K)\geq 0,\ \ \text{for}\ \ a.e\ \ x\in\R^3.
\]
}
\end{enumerate}
Replaced $|u|^{p-2}u$ by $a(x)f(u)$ in \eqref{Zhao}, Du $et. al$ \cite{Du} proved the existence and asymptotic behavior of solutions under conditions $(V_1)-(V_3)$ or $\widetilde{(V_1)}-(V_2)-(V_3)$ and some suitable assumptions $a(x)$ and $K(x)$, where $\lim_{t\to\infty}f(t)/t=l\in(0,+\infty)$.
There are many interesting works about the existence of positive solutions, positive ground states,
multiple solutions, sign-changing solutions and semiclassical states to \eqref{introduction2}, see e.g. \cite{Alves,Alves2,Ambrosetti,Azzollini,Azzollini2,He,Huang,Ruiz0,Ruiz,Shen,Shen2,Sun2,Zhang,Zhao2}
and their references therein.

The nonlinear fractional Schr\"{o}dinger-Poisson systems \eqref{mainequation1} come from the following fractional Schr\"{o}dinger equation
\begin{equation}\label{introduction3}
(-\Delta)^s u+V(x)u=f(x,u), \ \   x\in\R^N \\
\end{equation}
used to study the standing wave solutions $\psi(x, t) = u(x)e^{-i\omega t}$ for the equation
\[
i\hbar \frac{\partial \psi}{\partial t}=\hbar^2(-\Delta)^{\alpha}\psi+W(x)\psi-f(x,\psi)\ \ x\in\R^N,
\]
where $\hbar$ is the Planck's constant, $W:\R^N\to\R$ is an external potential and $f$ a suitable nonlinearity.
Since the fractional Schr\"{o}dinger equation appears in problems involving nonlinear optics, plasma physics and condensed matter physics, it is one of the main objects of the fractional quantum mechanic.
The equation \eqref{introduction3} has been firstly proposed by Laskin \cite{Laskin1,Laskin2} as a result of
expanding the Feynman path integral, from the Brownian-like to the L\'{e}vy-like quantum mechanical
paths. In their celebrated paper, Caffarelli-Silvestre \cite{Caffarelli} transform the nonlocal operator $(-\Delta)^\alpha$ to a Dirichlet-Neumann boundary value problem
for a certain elliptic problem with local differential operators defined on the upper half space.
This technique of Caffarelli-Silvestre is a valid tool to deal with the equations involving fractional operators in the respects of regularity and variational methods, please see \cite{Alves2,He} and their references for example.
When the conditions $(V_1)-(V_3)$ are satisfied, L. Yang and Z. Liu \cite{Liu3} proved the multiplicity and concentration of solutions for the following fractional Schr\"{o}dinger equation
\[
 (-\Delta)^s u+ \la V(x)u=f(x,u)+\alpha(x) |u|^{v-2}u, \ \   x\in\R^N,
\]
involving a $k$-order asymptotically linear term $f(x,u)$, where $s\in(0,1)$, $2s<N$, $1\leq k<2_s^*-1=\frac{N+2s}{N-2s}$ and $\alpha\in L^{\frac{v}{2-v}}(\R^N)$ with $1<v<2$.
Please see \cite{Ambrosio1,Ambrosio2,Molica,Felmer,Figueiredo} and their references for
some other related results on fractional Schr\"{o}dinger equation.

However similar results on the fractional Schr\"{o}dinger-Poisson systems are not as rich as the  Schr\"{o}dinger-Poisson system \eqref{introduction2}, especially there are very few results on the existence and concentration results with steep potential well. Very recently, K. Teng and R. Agarwal \cite{Teng3} considered the semiclassic case for the following fractional Schr\"{o}dinger-Poisson system
\[
    \left\{%
\begin{array}{ll}
    \epsilon^{2s}(-\Delta)^s u+  V(x)u+\phi u= K(x)f(u)+Q(x)|u|^{2_s^*-2}u, & x\in\R^3, \\
    \epsilon^{2t}(-\Delta)^t \phi=u^2, &  x\in\R^3,\\
\end{array}%
\right.
\]
under some appropriate conditions on $K(x)$, $Q(x)$ and $f\in C^1(\R^3)$ behaving like $|u|^{p-2}u$ with $4<p<2_s^*=\frac{6}{3-2s}$, where the existence and concentration of positive ground state solutions were obtained.
Other interesting results on fractional Schr\"{o}dinger-Poisson system can be found in \cite{Liu2,Murcia,Shen3,Teng2,Wei,Zhang2} and their references.

Motivated by all the works just described above, particularly by \cite{Zhao}, we prefer to investigate the existence and concentration results for \eqref{mainequation1} with steep potential well and more general nonlinearity. Since we are interested in {positive} solutions, without loss of generality, we assume that $f\in C^0(\R,\R)$ vanishes in $(-\infty,0)$ and satisfies the following conditions:
\vskip0.3cm
\begin{enumerate}[$(f_1)$]
\item \emph{$f\in C^0(\R, \R^+)$ and $f(z) =o(z)$ as $z \to 0$, where $\R^+=[0,+\infty)$;}
\end{enumerate}
 \vskip0.3cm
\begin{enumerate}[$(f_2)$]
 \item  \emph{$|f(z)|\leq C_0(1+|z|^{q-1})$ for some constants $C_0>0$ and $2<q<2^*_s=\frac{6}{3-2s}$;}
\end{enumerate}
\vskip0.3cm
\begin{enumerate}[$(f_3)$]
 \item  \emph{there exist a constant $\gamma>\frac{4s+2t}{s+t}$ such that $zf(z)-\gamma F(z)\geq 0$, where $F(z)=\int_{0}^{z}f(s)ds$.}
\end{enumerate}

Our main results are as follows:

\begin{theorem}\label{maintheorem1}
Let $s,t\in (0,1)$ satisfy $4s+2t>3$,
and assume that $(V_1)-(V_3)$, $(f_1)-(f_3)$, $K(x)\geq 0$ for all $x\in\R^3$ with $K(x)\in L^{\infty}(\R^3)\cap L^{\frac{6}{4s+2t-3}}(\R^3)$ with $s\geq t$. In addition, we assume the following conditions:
\begin{enumerate}[$(V_4)$]
  \item \emph{$V(x)$ is weakly differentiable and $(x,\nabla V)\in L^{\infty}(\R^3)\cup L^{\frac{3}{2s}}(\R^3)$ verifies the following inequality:
 \[
(s+t)(\gamma-2)V(x)+(x,\nabla V)\geq 0,
\]
where $(\cdot,\cdot)$ is the usual inner product in $\R^3$.}
\end{enumerate}
\vskip3mm
\begin{enumerate}[$(K)$]
  \item \emph{$K(x)$ is weakly differentiable and $(x,\nabla K)\in L^{\infty}(\R^3)\cup L^{\frac{6}{4s+2t-3}}(\R^3)$ satisfies the following inequality:
\[
\big[(s+t)\gamma-(4s+2t)\big]K(x)+2(x,\nabla K)\geq 0.
\]
}
\end{enumerate}
Then there exists $\Lambda>0$ such that the system \eqref{mainequation1} admits at least one positive ground state solution for all $\la>\Lambda$.
\end{theorem}

\begin{remark}
There are some remarks on Theorem \ref{maintheorem1} as follows:
\begin{enumerate}[$(1)$]
  \item  The hypothesi $K(x)\in L^{\frac{6}{4s+2t-3}}(\R^3)$ with $s\geq t$ is unnecessary if we restrict the work spaces to radially symmetric spaces, such as $H_r^s(\R^3)=\big\{u\in H^s(\R^3):u(x)=u(|x|)\big\}$. In other words if the work spaces are radially symmetric, we may have $\gamma\leq3$ which is an interesting phenomenon, where the positive constant $\gamma$ comes from $(f_3)$.
\end{enumerate}
\begin{enumerate}[$(2)$]
  \item  Compared with the conditions $(\widetilde{{V_4})}-\widetilde{({K})}$ in \cite{Zhao} and $(V_4)-(K)$ in our paper, we have to make a carefully analysis to the fractional Schr\"{o}dinger-Poisson system involving a more general nonlinearity. On the other hand, we always assume $q\in(2,2_s^*)$ in $(f_2)$, hence the assumptions $(V_4)-(K)$ are never redundant.
\end{enumerate}
\begin{enumerate}[$(3)$]
  \item  It should pointed out here that the above nonlinearity assumptions $(f_1)-(f_3)$ mainly were motivated by J. Sun and S. Ma \cite{Sun}. Compared with \cite{Sun}, some appropriate modifications were made to adapt the fractional Schr\"{o}dinger-Poisson system.
\end{enumerate}
\begin{enumerate}[$(4)$]
  \item  A typical example of the nonlinearity verifying the assumptions $(f_1)-(f_3)$ is given by $f (z)=|z|^{\gamma-2}z$ with $\gamma>\frac{4s+2t}{s+t}$.
\end{enumerate}
\end{remark}

\begin{remark}
Recently, K. Teng \cite{Teng2} and Shen-Yao \cite{Shen3} have considered the existence of ground state solutions to the following fractional Schr\"{o}dinger-Poisson system:
$$
  \left\{%
\begin{array}{ll}
    (-\Delta)^s u+ V(x)u+\phi u= |u|^{p-2}u+\mu |u|^{2_s^*-2}u, & x\in\R^3, \\
    (-\Delta)^t \phi=u^2, &  x\in\R^3,\\
\end{array}%
\right.
$$
with $\mu\ge0$ and $2<p<2_s^*$ under suitable assumptions of $V(x)$. The two papers above were required to meet condition $2s+2t>3$,  which is more restricted than the condition $4s+2t>3$ in this paper if $f(u)$ behaves like $|u|^{q-2}u$ with $2<q<2_s^*$. In fact, we remark that by  the techniques here, the condition $2s+2t>3$ can be improved to the inequality $4s+2t>3$.
\end{remark}

Inspired by the results in \cite{Bartsch,Du,Jiang,Liu3,Zhao}, we get the following concentration result:

\begin{theorem}\label{maintheorem2}
Let $(u_\la,\phi_{u_\la})$ be the nontrivial solutions obtained in Theorem \ref{maintheorem1}, then $u_\la \to u_0$ in $H^s(\R^3)$ (see Section 2 below) and $\phi_{u_\la} \to \phi_{u_0}$ in $D^{t,2}(\R^3)$ (see Section 2 below) as $\la\to+\infty$, where $u_0\in H^s_0(\Omega)$ is a nontrivial solution to
 \begin{equation}\label{mainequation2}
    \left\{%
\begin{array}{ll}
    (-\Delta)^s u+c_t\bigg(\big(K(x)u^2\big)*\frac{1}{|x|^{3-2t}}\bigg) K(x)\phi u= f(u), & x\in \Omega, \\
    u=0, &  \text{on}\ \  \partial \Omega.\\
\end{array}%
\right.
\end{equation}
Note that $c_t>0$ is a constant form \eqref{transform2} below.
\end{theorem}

Now we give our main ideas for the proofs of Theorem \ref{maintheorem1} and \ref{maintheorem2}. It is not simple to verify that $I_\la$ (see Section 2) possesses a Mountain-pass geometry in the usual way because the Ambrosetti-Rabinowitz
type condition ($(AR)$ in short):
\begin{enumerate}[$(AR)$]
\item \ \ \ \ \ \ \ \   There exists $\eta>4$ such that $0 < \eta F(t)\leq f(t)t$ for all $t \neq 0$
\end{enumerate}
or 4-superlinear at infinity in the sense that
\[
(F)\ \  \ \ \ \ \ \ \ \ \ \ \ \ \ \   \ \  \ \ \ \ \ \ \ \ \ \ \ \ \ \
\lim\limits_{|t|\to \infty}\frac{F(t)}{|t|^4}=+\infty.\ \  \ \ \ \ \ \ \ \ \ \ \ \ \ \ \ \  \ \ \ \ \ \ \ \ \ \ \ \ \ \ \ \  \ \ \ \ \ \ \ \ \ \ \ \ \ \ \ \  \ \ \ \ \ \ \ \ \ \ \ \ \ \   \ \  \ \ \ \ \ \ \ \ \ \ \ \ \ \    \ \  \ \ \ \ \ \ \ \ \ \ \ \ \ \
\]
does not always hold.
Furthermore, even
if a $(PS)$ sequence has been obtained, it is difficult to prove its boundedness since
the nonlinearity $f(u)$ behaving like $|u|^{q-2}u$ with $2<q<2^*_s$ results in
neither the weaker condition $(AR)_4$ ($\eta=4$ in $(AR)$) nor the condition
\begin{enumerate}[$(M)$]
  \item The map $t \to
\frac{f(t)}{t^3}$ is positive for $t\neq 0$, strictly decreasing on $(-\infty,0)$ and
strictly increasing on $(0,+\infty)$.
\end{enumerate}
works yet. To overcome this difficulties, motivated by \cite{Zhao2}, we use an indirect approach (see Proposition \ref{proposition}) developed by L. Jeanjean \cite{Jeanjean3} to get a bounded $(PS)$ sequence.
Though a bounded $(PS)$ sequence can be constructed, another difficulty on the lack of compactness of the Sobolev embedding $H^s(\R^3)\hookrightarrow L^r(\R^3)$ with $2\leq r\leq 2_s^*$ occurs and the $(PS)$ condition seems to be hard to verify because we do not assume the potential $V(x)$ and the weight function $K(x)$ to be radially symmetric. To solve it, we assume $K(x)\in L^{\frac{6}{4s+2t-3}}(\R^3)$ with $s\geq t$ to recover the compactness and then to prove the $(PS)$ condition. So far, we can prove the Theorem \ref{maintheorem1} and \ref{maintheorem2} step by step.

The paper is organized as follows. In Section 2, the function spaces will be introduced and then we provide several lemmas, which are crucial in proving
our main results. In Section 3, the proof of Theorem \ref{maintheorem1} is obtained. The concentration result of Theorem \ref{maintheorem2} will be proved in Section 4.
\\\\
\textbf{Notations.} Throughout this paper we shall denote by $C$ and $C_i$ ($i=1, 2,\cdots$) for various positive constants whose exact value may change from lines to lines but are not essential to the analysis of problem. $L^p(\R^3)$ $(1\leq p\leq+\infty)$ is the usual Lebesgue space with the standard norm $|u|_p$.
We use $``\to"$ and $``\rightharpoonup"$ to denote the strong and weak convergence in the related function space, respectively. The symbol $``\hookrightarrow"$ means a function space is continuously imbedding into another function space. The Lebesgue measure of a Lebesgue measurable set $E$ in $\R^3$ is $|E|$.
For any $\rho>0$ and any $x\in \R^3$, $B_\rho(x)$ denotes the ball of radius $\rho$ centered at $x$, that is, $B_\rho(x):=\{y\in \R^3:|y-x|<\rho\}$.

Let $(X,\|\cdot\|)$ be a Banach space with its dual space $(X^{-1},\|\cdot\|_{*})$, and $\Phi$ be its functional on $X$. The Palais-Smale sequence at level $c\in\R$ ($(PS)_c$ sequence in short) corresponding to $\Phi$ assumes that $\Phi(x_n)\to c$ and $\Phi^{\prime}(x_n)\to 0$ as $n\to\infty$, where $\{x_n\}\subset X$. If for any $(PS)_c$ sequence $\{x_n\}$ in $X$, there exists a subsequence $\{x_{n_{k}}\}$ such that $x_{n_{k}}\to x_0$ in $X$ for some $x_0\in X$, then we say that the functional $\Phi$ satisfies the so called $(PS)_c$ condition.
\section{Variational settings and preliminaries}
In this section, we first bring in some necessary variational settings for system \eqref{mainequation1} and the complete
introduction to the fractional Sobolev spaces can be found in \cite{Nezza}. Recalling that the fractional Sobolev space $W^{\alpha,p}(\R^N)$ is defined for any
$p\in[1, +\infty)$ and $\alpha\in(0, 1)$ as follows
\[
W^{\alpha,p}(\R^N)=\bigg\{u\in L^p(\R^N):\int_{\R^N}\int_{\R^N}\frac{|u(x)-u(y)|^p}{|x-y|^{N+\alpha p}}dxdy<+\infty\bigg\}
\]
equipped with the natural norm
\[
\|u\|_{W^{\alpha,p}(\R^N)}=\bigg(\int_{\R^N}\int_{\R^N}\frac{|u(x)-u(y)|^p}{|x-y|^{N+\alpha p}}dxdy+\int_{\R^N}|u|^pdx\bigg)^{\frac{1}{p}}.
\]
In particular, if $p=2$, the fractional Sobolev space $W^{\alpha,2}(\R^N)$ is simply denoted by $H^\alpha(\R^N)$. As we all know, the fractional Sobolev space $H^{\alpha}(\R^N)$ can be also described by the Fourier transform, that is,
\[
H^{\alpha}(\R^N)=\bigg\{u\in L^2(\R^N):\int_{\R^N}|\xi|^{2\alpha}|\widehat{u}(\xi)|^2+|\widehat{u}(\xi)|^2d\xi<+\infty \bigg\},
\]
where $\hat{u}$ denotes the usual Fourier transform of $u$. When we take the definition of the fractional Sobolev space $H^{\alpha}(\R^N)$ by the Fourier transform, the inner product and the norm for $H^{\alpha}(\R^N)$ are defined as
\[
(u.v)_0=\int_{\R^N}|\xi|^{2\alpha}\widehat{u}(\xi)\widehat{v}(\xi)+\widehat{u}(\xi)\widehat{v}(\xi)d\xi
\]
and
\[
\|u\|_{H^{\alpha}(\R^N)}=\bigg(\int_{\R^N}|\xi|^{2\alpha}|\widehat{u}(\xi)|^2+|\widehat{u}(\xi)|^2d\xi\bigg)^{\frac{1}{2}}.
\]
Following from Plancherel's theorem, one has $|u|_2=|\widehat{u}|_2$ and $|(-\Delta)^{\frac{\alpha}{2}}u|_2=||\xi|^\alpha\widehat{u}|_2$. Hence
\begin{equation}\label{definition}
  \|u\|_{H^{\alpha}(\R^N)}=\bigg(\int_{\R^N}|(-\Delta)^{\frac{\alpha}{2}}u|^2+|u|^2dx\bigg)^{\frac{1}{2}},\ \ \forall~u\in H^{\alpha}(\R^N).
\end{equation}
As a consequence of \cite[Proposition 3.4 and Proposition 3.6]{Nezza}, one has
\[
 |(-\Delta)^{\frac{\alpha}{2}}u|_2=\bigg(\int_{\R^N}|\xi|^{2\alpha}|\widehat{u}(\xi)|^2d\xi\bigg)^{\frac{1}{2}}
=\bigg(\frac{1}{C_N(\alpha)}\int_{\R^N}\int_{\R^N}\frac{|u(x)-u(y)|^2}{|x-y|^{N+2\alpha}}dxdy\bigg)^{\frac{1}{2}}.
\]
which reveals that the norm given by \eqref{definition} makes sense for the fractional Sobolev space. Meanwhile
the homogeneous fractional Sobolev space $D^{\alpha,2}(\R^N)$ is defined by
\[
D^{\alpha,2}(\R^N)=\bigg\{ u\in L^{2^*_\alpha}(\R^N): |\xi|^\alpha \widehat{u}(\xi)\in L^{2^*_\alpha}(\R^N)  \bigg\}
\ \  \text{with}\ \ 2^*_\alpha=\frac{2N}{N-2\alpha}\ \ \text{and} \ \  N\geq 3.
\]
which is the completion of $C^{\infty}_0(\R^N)$ under the norm
\[
\|u\|_{D^{\alpha,2}(\R^N)}=\bigg(\int_{\R^N}|(-\Delta)^{\frac{\alpha}{2}}u|^2dx\bigg)^{\frac{1}{2}}=
\bigg(\int_{\R^N}|\xi|^{2\alpha}|\widehat{u}(\xi)|^2d\xi\bigg)^{\frac{1}{2}}.
\]

The following fractional Sobolev embedding theorems are necessary.
\begin{lemma}\label{imbedding}
 (see \cite{P. L. Lions}) For any $\alpha\in (0,\frac{N}{2})$, $H^\alpha(\R^N)$ is continuously embedded into $L^r(\R^N)$ for $r\in[2,2^*_\alpha]$ and compactly embedded into $L^r_{\text{loc}}(\R^N)$ for $r\in [1, 2^*_\alpha)$.
\end{lemma}
 As a direct consequence of Lemma \ref{imbedding}, there are constants $C_r> 0$ such that
\begin{equation}\label{Sobolev1}
\|u\|_{H^\alpha(\R^N)} \leq C_r|u|_r, \ \  \forall~u\in H^\alpha(\R^N)\ \ \text{and}\ \ 2\leq r\leq 2^{*}_\alpha.
\end{equation}
Also there exists a best constant $S_\alpha>0$ (see \cite{Cotsiolis}) such that
\begin{equation}\label{Sobolev2}
  S_\alpha=\inf_{u\in D^{\alpha,2}(\R^N)\setminus \{0\}}\frac{\int_{\R^N}|(-\Delta)^{\frac{\alpha}{2}}u|^2dx}{\big(\int_{\R^N}|u|^{2^*_\alpha}dx\big)^{\frac{2}{2^*_\alpha}}}.
\end{equation}

In this paper, for $s,t\in(0,1)$ we restrict the work spaces in dimension $N=3$ and let
\[
E\triangleq\bigg\{u\in H^s(\R^3):\int_{\R^3}V(x)u^2dx<+\infty\bigg\}
\]
be endowed with the inner product and the norm
\[
(u,v)= \int_{\R^3}(-\Delta)^{\frac{s}{2}}u(-\Delta)^{\frac{s}{2}}v+V(x)uvdx, \ \ \|u\|=\bigg(\int_{\R^3}|(-\Delta)^{\frac{s}{2}}u|^2+V(x)u^2dx\bigg)^{\frac{1}{2}}
\]
for any $u,v\in E$. By using the assumptions $(V_1)-(V_2)$ and \eqref{Sobolev2}, one has
\begin{align*}
   \int_{\R^3}u^2dx &=\int_{\{V\geq c\}}u^2dx+\int_{\{V<c\}}u^2dx   \\
    & \leq \frac{1}{c}\int_{\{V\geq c\}}V(x)u^2dx+\big|\{V<c\}\big|^{\frac{2_s^*-2}{2_s^*}}\bigg(\int_{\{V<c\}}|u|^{2_s^*}dx\bigg)^{\frac{2}{2_s^*}}\\
    &\leq \max\bigg\{\frac{1}{c},\big|\{V<c\}\big|^{\frac{2_s^*-2}{2_s^*}}\bigg\}\|u\|^2
\end{align*}
which implies that the imbedding $E\hookrightarrow H^s(\R^3)$ is continuous. Thus by \eqref{Sobolev1} there exists $d_r>0$ such that
\begin{equation}\label{Sobolev3}
|u|_r \leq d_r\|u\|, \ \  \forall~u\in E\ \ \text{and}\ \ 2\leq r\leq 2^{*}_s.
\end{equation}
For any $\la>0$, we let $E_\la\triangleq(E,\|\cdot\|_\la)$ and the inner product and norm are
\[
(u,v)_\la= \int_{\R^3}(-\Delta)^{\frac{s}{2}}u(-\Delta)^{\frac{s}{2}}v+\la V(x)uvdx, \ \ \|u\|_\la=\bigg(\int_{\R^3}|(-\Delta)^{\frac{s}{2}}u|^2+\la V(x)|u|^2dx\bigg)^{\frac{1}{2}}.
\]
Obviously, $\|u\|\leq \|u\|_\la$ if $\la\geq 1$. The following facts
\[
\int_{\{V<c\}}|u|^2dx  \leq  \big|\{V<c\}\big|^{\frac{2_s^*-2}{2_s^*}}|u|^2_{ 2_s^*}
   \stackrel{\mathrm{\eqref{Sobolev2}}}{\leq} \big|\{V<c\}\big|^{\frac{2_s^*-2}{2_s^*}}S_s^{-1}\|u\|^2_\la
\]
and
\[
 \int_{\{V\geq c\}}|u|^2dx\leq \frac{1}{\la c}\int_{\{V\geq c\}}\la V(x)|u|^2dx\leq \frac{1}{\la c}\int_{\R^3}\la V(x)|u|^2dx \leq \frac{1}{\la c}\|u\|^2_\la
\]
give us that for any $r\in[2,2_s^*]$
\begin{eqnarray*}
 \int_{\R^3}|u|^rdx &\leq&\bigg(\int_{\R^3}|u|^2dx\bigg)^{\frac{2_s^*-r}{2_s^*-2}}
 \bigg(\int_{\R^3}|u|^{2_s^*}dx\bigg)^{\frac{r-2}{2_s^*-2}}  \\
   &\stackrel{\mathrm{\eqref{Sobolev2}}}{\leq} & \bigg(2\max\Big\{S_s^{-1}\big|\{V<c\}\big|^{\frac{2_s^*-2}{2_s^*}},\frac{1}{\la c}\Big\}\|u\|^2_\la\bigg)^{\frac{2_s^*-r}{2_s^*-2}}
   \bigg(S_s^{-\frac{2_s^*}{2}}\|u\|^{2_s^*}_\la\bigg)^{\frac{r-2}{2_s^*-2}}.
\end{eqnarray*}
Hence for any $r\in[2,2_s^*]$, we have that
\begin{equation}\label{Lr}
 \int_{\R^3}|u|^rdx \leq \bigg(2\big|\{V<c\}\big|\bigg)^{\frac{2_s^*-r}{2_s^*}}
S_s^{-\frac{r}{2}}\|u\|^{r}_\la\ \ \text{whenever} \ \  \la\geq
c^{-1}\big|\{V<c\}\big|^{-\frac{2_s^*-2}{2_s^*}}  S_s.
\end{equation}

It is similar to the usual Sch\"{o}rdinger-Poisson system that the system \eqref{mainequation1} can reduce to be a single equation. Indeed, using the H\"{o}lder inequality,
for every $u\in H^{s}(\R^3)$ and $v\in D^{t,2}(\R^3)$, one has
\begin{eqnarray}\label{transform0}
\nonumber\int_{\R^3}K(x)u^2vdx &\leq& |K|_{\frac{6}{4s+2t-3}} \bigg(\int_{\R^3}|u|^{\frac{6}{3-2s}}dx\bigg)^{\frac{3-2s}{3}}\bigg(\int_{\R^3}|v|^{\frac{6}{3-2t}}dx\bigg)^{\frac{3-2t}{6}}\\
 &\leq&|K|_{\frac{6}{4s+2t-3}}S_s^{-1} S_t^{-\frac{1}{2}}\|u\|^2_{D^{t,2}(\R^3)}\|v\|_{D^{t,2}(\R^3)}
\leq C\|u\|^2 \|v\|_{D^{t,2}(\R^3)},
\end{eqnarray}
where we use the fact that $E\hookrightarrow H^{s}(\R^3)\hookrightarrow L^{2_s^*}(\R^3)$.
For any $u\in H^s(\R^3)$, one can use the Lax-Milgram theorem and then
there exists a unique $\phi_u^t\in D^{t,2}(\R^3)$ such that
\begin{equation}\label{transform1}
 \int_{\R^3}(-\Delta)^{t}\phi_u^tvdx=\int_{\R^3}(-\Delta)^{\frac{t}{2}}\phi_u^t (-\Delta)^{\frac{t}{2}}vdx=\int_{\R^3}K(x)u^2vdx,\ \  \forall~v\in D^{t,2}(\R^3).
\end{equation}
In other words, $\phi_u^t$ satisfies the Poisson equation
$$
(-\Delta)^t\phi_u^t=K(x) u^2,\ \ x\in\R^3
$$
and we can write it an integral expression, that is,
\begin{equation}\label{transform2}
  \phi_u^t(x)=c_t\int_{\R^3}\frac{K(x)u^2(y)}{|x-y|^{3-2t}}dx,\ \ x\in\R^3,
\end{equation}
which is called $t$-Riesz potential, where
\[
c_t=\pi^{-\frac{3}{2}}2^{-{2t}}\frac{\Gamma(\frac{3}{2}-2t)}{\Gamma(t)}.
\]
It follows from \eqref{transform2} that $\phi_u^t(x)\geq 0$ for all $x\in\R^3$. Taking $v=\phi_u^t$ in \eqref{transform0} and \eqref{transform1}, we derive
\begin{equation}\label{transform4}
 \|\phi_u^t\|_{D^{t,2}(\R^3)}\leq C\|u\|^2.
\end{equation}

Substituting \eqref{transform2} into \eqref{mainequation1}, we can rewrite \eqref{mainequation1} in the
following equivalent form
\begin{equation}\label{mainequation5}
(-\Delta)^{s}u+\la V(x)u+K(x)(-\Delta)^{t}\phi_u^tu= f(u),\ \ x\in\R^3.
\end{equation}
The energy functional $I_\la:H^s(\R^3)\to\R$ associated to the problem \eqref{mainequation5} is given by
\begin{equation}\label{functional}
 I_\la(u)=\frac{1}{2}\|u\|^2_\la +\frac{1}{4}\int_{\R^3}K(x)\phi_u^t u^2 dx
  -\int_{\R^3}F(u)dx.
\end{equation}
If we take $v=\phi_u^t$ in \eqref{transform0} and \eqref{transform1} again, we get
\[
\int_{\R^3}K(x)\phi_u^tu^2dx\leq C\|u\|^2 \|\phi_u^t\|_{D^{t,2}(\R^3)}\stackrel{\mathrm{\eqref{transform4}}}{\leq} C\|u\|^4.
\]
It is therefore that $I_\la(u)$ is well-defined and $I_\la\in C^1(E_\la,\R)$ by \eqref{functional} ({see \cite{Willem} for details}), moreover its differential is
\[
\langle I^{\prime}_\la(u),v\rangle= \int_{\R^3}(-\Delta)^{\frac{s}{2}}u (-\Delta)^{\frac{s}{2}}v dx+ \int_{\R^3}\la V(x)uvdx+ \int_{\R^3}K(x)\phi_u^t uv dx  -\int_{\R^3}f(u)vdx
\]
for any $u,v\in E_\la$. It is clear that if $u$ is a critical points of $I_\la$, then the pair $(u,\phi_u^t)$
is a solution of system \eqref{mainequation1}.

Before giving the necessary lemmas for this paper, it is important to stress that the conditional assumptions in Theorem \ref{maintheorem1} and Theorem \ref{maintheorem2} are always true for simplicity. By simple calculations, we can deduce from $(f_1)$ and $(f_2)$ that
\begin{equation}\label{growth1}
  |f(u)|\leq \epsilon|u|+C_\epsilon |u|^{q-1}\ \ \text{and}\ \ |F(u)|\leq \epsilon u^2+C_\epsilon |u|^{q}.
\end{equation}
It follows from $(f_1)$ and $(f_2)$ that there exists a constant $C>0$ such that
 \begin{equation}\label{growth2}
F(u)\geq C|u|^\gamma.
\end{equation}

\begin{lemma}\label{aaaaa}
Assume $K(x)\in L^{\frac{6}{4s +2t-3}}(\R^3)$ with $4s+2t>3$ and $s\geq t$, then the following properties are true:
\begin{enumerate}[$(a)$]
\item If $u\in H^s(\R^3)$ and we set $u_\theta(x):=\theta^{s+t}u(\theta x)$ for $\theta\in\R^+$, then
\[
\int_{\R^3}\phi_{u_\theta}^tu_\theta^2dx=\theta^{4s+2t-3}\int_{\R^3}\phi_{u}^tu^2dx<+\infty.
\]
 \item $\phi_{u(\cdot+y)}^t=\phi_{u}^t(x+y)$.
 \item If $u_n\rightharpoonup u$ in $H^s(\R^3)$, then $\phi_{u_n}^t\rightharpoonup \phi_{u}^t$ in $D^{t,2}(\R^3)$.
\end{enumerate}
\end{lemma}

\begin{proof}
$(a)$ Since $4s+3t>3$, then $u\in H^s(\R^3)\hookrightarrow L^{\frac{12}{3+2t}}(\R^3)$ and thus
\[
\int_{\R^3}\phi_{u}^tu^2dx\leq |\phi_u^t|_{2_t^*}|u|_{\frac{12}{3+2t}}^2\stackrel{\mathrm{\eqref{Sobolev2}}}{\leq}
S_t^{-\frac{1}{2}}\|\phi_u^t\|_{D^{t,2}(\R^3)}|u|_{\frac{12}{3+2t}}^2\stackrel{\mathrm{\eqref{transform4}}}{\leq}
C\|u\|^2|u|_{\frac{12}{3+2t}}^2<+\infty.
\]
By means of \eqref{transform2}, one has
\begin{eqnarray*}
\int_{\R^3}\phi_{u_\theta}^tu_\theta^2dx&=& c_t\int_{\R^3}\int_{\R^3}\frac{u^2_\theta(y)u^2_\theta(x)}{|x-y|^{3-2t}}dydx \\
  &=& \theta^{4s+4t}\theta^{3-2t}\theta^{-6}c_t\int_{\R^3}\int_{\R^3}\frac{u^2(y)u^2(x)}{|x-y|^{3-2t}}dydx \\
   &=&\theta^{4s+2t-3}\int_{\R^3}\phi_{u}^tu^2dx.
\end{eqnarray*}

$(b)$ It is a direct consequence of \eqref{transform2}.

$(c)$ If $u_n\rightharpoonup u$ in $H^s(\R^3)$, by Lemma \ref{imbedding} and $\frac{3}{3-2s}\in(1,\frac{6}{3-2s})$, there exists a subsequence still denoted by itself such that $u_n\to u$ in $L_{\text{loc}}^{\frac{3}{3-2s}}(\R^3)$. Since $s\geq t$, then $\frac{6}{3-2t}\in(2,\frac{6}{3-2s}]$ and hence $|u_n+u|$ is uniformly bounded in $L^{\frac{6}{3-2t}}(\R^3)$. On the other hand for any $\varphi\in C_0^{\infty}(\R^3)$, then $\varphi\in L^{\infty}(\R^3)$ and we have that
\[
\bigg|\int_{\R^3} K(x)({u_n}^2-u^2)\varphi dx\bigg|\leq  |\varphi|_\infty |K|_{\frac{6}{4s+2t-3}}
|u_n+u|_{\frac{6}{3-2t}}\bigg( \int_{\supp \varphi}|u_n-u|^{\frac{3}{3-2s}}dx\bigg)^{{\frac{3-2s}{3}}}\to 0,
\]
where $\supp \varphi$ denotes the support of $\varphi$.
Since $C_0^{\infty}(\R^3)$ is dense in $H^s(\R^3)$, then the above formula shows that $(c)$ is true.
 \end{proof}

The following lemma will play an vital role in recovering the compactness for the $(PS)$ sequence, which is similar to the well-known Br\'{e}zis-Lieb lemma \cite{Brezis}.

\begin{lemma}
Assume $K(x)\in L^{\frac{6}{4s +2t-3}}(\R^3)$ with $4s+2t>3$ and $s\geq t$, if $u_n\rightharpoonup u$ in $H^s(\R^3)$ and $u_n\to u$ $a.e.$ in $\R^3$, then we have that
 \begin{equation}\label{weak1}
\int_{\R^3}K(x)\phi_{u_n}^tu_n^2dx-\int_{\R^3}K(x)\phi_{u}^tu^2dx\to 0,
\end{equation}
and
\begin{equation}\label{weak2}
  \int_{\R^3}K(x)\phi_{u_n}^tu_n\varphi dx-\int_{\R^3}K(x)\phi_{u}u\varphi dx\to 0
\end{equation}
for any $\varphi \in C^\infty_0(\R^3)$.
\end{lemma}
\begin{proof}
We point out here that the proof of the case $s=t=1$ for this lemma can be found in \cite{Zhao}, which can be viewed as a special one in our paper. Since $u\in H^s(\R^3)\hookrightarrow L^{\frac{6}{3-2s}}(\R^3)$, then one has
\[
\int_{\R^3}|K|^{\frac{6}{3+2t}} |u|^{\frac{12}{3+2t}}dx\leq\bigg(\int_{\R^3}|K|^{\frac{6}{4s+2t-3}}dx\bigg)^{\frac{4s+2t-3}{3+2t}}
 \bigg(\int_{\R^3}|u|^{\frac{6}{3-2s}}dx\bigg)^{\frac{6-4s}{3+2t}}<+\infty
\]
which implies that $Ku^2\in L^{\frac{6}{3+2t}}(\R^3)$. By $(c)$ of Lemma \ref{aaaaa} and \eqref{Sobolev2}, one has $\phi_{u_n}^t\rightharpoonup\phi_u^t$ in $L^{\frac{6}{3-2t}}(\R^3)$ and thus
\[
A_1\triangleq\int_{\R^3}K(x)\phi_{u_n}^tu^2dx- \int_{\R^3}K(x)\phi_{u}^tu^2dx\to 0.
\]

On the other hand, $u_n\rightharpoonup u\in H^s(\R^3)$ gives that $|u_n-u|\rightharpoonup 0$ in $L^{\frac{6}{3-2s}}(\R^3)$ and then $|u_n-u|^{\frac{3}{s+t}}\rightharpoonup 0$ in $L^{\frac{2(s+t)}{3-2s}}(\R^3)$. Since $|K|^{\frac{3}{s+t}}\in L^{\frac{2(s+t)}{4s+2t-3}}(\R^3)$, then
\[
\int_{\R^3} |K|^{\frac{3}{s+t}}|u_n-u|^{\frac{3}{s+t}}dx\to0
\]
which shows that
\begin{eqnarray*}
|A_2| &\triangleq&\bigg|\int_{\R^3}K(x)\phi_{u_n}^tu_n^2dx-\int_{\R^3}K(x)\phi_{u_n}^tu^2dx\bigg|   \\
   &\leq& \bigg(\int_{\R^3}|K|^{\frac{3}{s+t}}|u_n-u|^{\frac{3}{s+t}}dx\bigg)^{\frac{s+t}{3}}
\bigg(\int_{\R^3}|\phi_{u_n}^t|^{\frac{6}{3-2t}}dx\bigg)^{\frac{3-2t}{6}}
\bigg(\int_{\R^3}|u_n+u|^{\frac{6}{3-2s}}dx\bigg)^{\frac{3-2s}{6}}  \\
   &\leq&C  \bigg(\int_{\R^3}|K|^{\frac{3}{s+t}}|u_n-u|^{\frac{3}{s+t}}dx\bigg)^{\frac{s+t}{3}}\to 0.
\end{eqnarray*}
Consequently, we have that
\[
\int_{\R^3}K(x)\phi_{u_n}^tu_n^2dx-\int_{\R^3}K(x)\phi_{u}^tu^2dx=A_1+A_2\to 0.
\]
The proof of formula \eqref{weak2} is totally same as that of \eqref{weak1}, so we omit it.
\end{proof}

As described in Section 1, it is difficult for us to construct a bounded $(PS)$ sequence because the conditions $(AR)$, $(M)$ and $(F)$ do not hold. Thanks to the following well-known proposition, we can do it successfully.

\begin{proposition}\label{proposition}
(See \cite[Theorem 1.1 and Lemma 2.3]{Jeanjean2}) Let $(X,\|\cdot\|)$ be a Banach space and $T\subset R^+$ be an interval, consider a family of $C^1$ functionals on $X$ of the form
$$
\Phi_{\mu}(u)=A(u)-\mu B(u),\ \  \forall \mu\in T,
$$
with $B(u)\geq 0$ and either $A(u)\to +\infty$ or $B(u)\to +\infty$ as $\|u\|\to +\infty$. Assume that there are two points $v_1,v_2\in X$ such that
\[
 c_{\mu}=\inf_{\gamma\in \Gamma}\sup_{\theta\in [0,1]}\Phi_{\mu}(\gamma(\theta))>\max\{\Phi_\mu(v_1),\Phi_\mu(v_1)\},
\ \  \forall \mu\in T,
\]
where
\[
  \Gamma=\{\gamma\in C([0,1],X):\gamma(0)=v_1, \gamma(1)=v_2\}.
\]
 Then, for almost every $\mu\in T$, there is a sequence $\{u_n(\mu)\}\subset X$ such that
\begin{enumerate}[$(a)$]
  \item $\{u_n(\mu)\}$ is bounded in $X$;
\end{enumerate}

\begin{enumerate}[$(b)$]
  \item $\Phi_\mu(u_n(\mu))\to c_\mu$ and $\Phi^{\prime}_\mu(u_n(\mu))\to 0$;
\end{enumerate}

\begin{enumerate}[$(c)$]
  \item the map $\mu\to c_\mu$ is non-increasing and left continuous.
\end{enumerate}
\end{proposition}

Letting $T=[\delta,1]$, where $\delta\in (0,1)$ is a positive constant. To apply Proposition \ref{proposition}, we will introduce a family of $C^1$ functionals on $X=E_\la$ with the form
\begin{equation}\label{IVK}
  I_{\lambda,\mu}(u)= \frac{1}{2}\int_{\R^3}|(-\Delta)^{\frac{s}{2}} u|^2+\la V(x)|u|^2dx+\frac{1}{4}\int_{\R^3}K(x)\phi_{u}^tu^2dx
 -\mu\int_{\R^3}F(u)dx.
\end{equation}
 Then let $I_{\lambda,\mu}(u)=A(u)-\mu B(u)$, where
$$
A(u)=\frac{1}{2}\int_{\R^3} |(-\Delta)^{\frac{s}{2}} u|^2+\la V(x)|u|^2dx+\frac{1}{4}\int_{\R^3}K(x)\phi_{u}^tu^2dx\to+\infty \ \  \text{as}\ \ \|u\|_\la\to+\infty,
$$
and
$$
B(u)=\int_{\R^3}F(u)dx\geq 0.
$$
It is clear that $I_{\lambda,\mu}$ is a well-defined $C^1$ functional on the space $E_\la$, and for all $u,v\in E_\la$, one has
\[
 \langle I^{\prime}_{\lambda,\mu}(u),v\rangle =\int_{\R^3}(-\Delta)^{\frac{s}{2}}u(-\Delta)^{\frac{s}{2}}v+\la V(x)uvdx +\int_{\R^3}K(x)\phi_u^tuvdx
   -\mu\int_{\R^3}f(u)vdx.
\]

We now in a position to verify the Mountain-pass geometry for the functional $I_{\lambda,\mu}$.

 \begin{lemma}\label{2Mountpass}
The functional $I_{\lambda,\mu}$ possesses a Mountain-pass geometry, that is,
\begin{enumerate}[$(a)$]
  \item there exists $v\in E\setminus\{0\}$ independent of $\mu$ such that $I_{\lambda,\mu}(v)\leq 0$ for all $\mu\in [\delta,1]$;
  \item $c_{\la,\mu}\triangleq\inf_{\gamma\in\Gamma}\sup_{\theta\in [0,1]}I_{\lambda,\mu}(\gamma(\theta))>\max\{I_{\lambda,\mu}(0),I_{\lambda,\mu}(v)\}$ for all $\mu\in [\delta,1]$, where
$$
\Gamma=\{\gamma\in C([0,1],E):\gamma(0)=0, \gamma(1)=v\}.
$$
\item there exists $M_0>0$ independent of $\la$ and $\mu$ such that $c_{\la,\mu}\leq M_0$
\end{enumerate}
\end{lemma}

\begin{proof}
$(a)$ $\Omega$ is an open nonempty set in $\R^3$ by $(V_3)$, without loss of generality, we assume $0\in \Omega$ and then there exists $\rho_0>0$ such that $B_{\rho_0}(0)\subset \Omega$. Let $\psi\in C_0^\infty(\R^3)$ satisfy that $\supp \psi\subset B_{\rho_0}(0)$ and $\psi_\theta=\theta^{s+t}\psi(\theta x)$, then $\supp \psi_\theta\subset B_{\rho_0}(0)$ if $\theta>1$. Hence for $\theta>1$ and $V(x)\equiv 0$ in $\Omega$, one has
\[
0\leq \int_{\R^3}V(x)\psi_\theta^2 dx= \int_{\supp \psi}V(x)\psi_\theta^2 dx
\leq \int_{B_{\rho_0}(0)}V(x)\psi_\theta^2 dx\leq \int_{\Omega}V(x)\psi_\theta^2 dx=0.
\]
In view of Lemma \ref{aaaaa} $(a)$ and \eqref{growth2}, we have that
 \begin{eqnarray}\label{Mountainpass}
\nonumber I_{\la,\delta}(\psi_\theta)&\leq& \frac{\theta^{4s+2t-3}}{2}\int_{\R^3}|(-\Delta)^{\frac{s}{2}}\psi|^2dx
 +\theta^{4s+2t-3}|K|_\infty\int_{\R^3}\phi_{\psi}^t\psi^2dx-\theta^{-3}\delta\int_{\R^3}F(\theta^{s+t}\psi)dx\\
\nonumber &\leq& \frac{\theta^{4s+2t-3}}{2}\int_{\R^3}|(-\Delta)^{\frac{s}{2}}\psi|^2dx
 +\theta^{4s+2t-3}|K|_\infty\int_{\R^3}\phi_{\psi}^t\psi^2dx-\theta^{(s+t)\gamma-3}\delta\int_{\R^3}|\psi|^\gamma dx\\
 &\to&-\infty
\end{eqnarray}
as $\theta\to+\infty$ because $\gamma>\frac{4s+2t}{s+t}$. Therefore we can take $v=\psi_{\theta_0}$ for some sufficiently large $\theta_0$, thus $I_{\la,\mu}(v)\leq I_{\la,\delta}(v)<0$ for all $\mu\in[\delta,1]$.

$(b)$ By means of \eqref{Sobolev3} and \eqref{growth1}, one has
\[
I_{\lambda,\mu}(u)\geq \frac{1}{2}\|u\|^2_\la-\epsilon\|u\|^{2}_\la-C_\epsilon\|u\|^{q}_\la.
\]
Let $\epsilon=\frac{1}{4}$, then $I_{\lambda,\mu}(u)>0$ when $q>2$ and $\|u\|_\la=\rho>0$ small.

$(c)$ let $\widetilde{\gamma}(\theta)=v_\theta=\theta^{s+t}v(\theta x)$, where $v$ is given by $(a)$. Recalling the definition of $e$ and $\Gamma$ given by $(b)$, one has $\widetilde{\gamma}\in \Gamma$. Therefore we have that
\[
  c_\la\triangleq\inf_{\gamma\in \Gamma}\max_{\theta\in [0,1]}I_{\la,\mu}(\gamma(\theta))\leq \max_{\theta\in [0,1]}I_{\la,\mu}(v_\theta)\leq \max_{\theta\geq 0}I_{\la,\delta}(\psi_\theta).
\]
Using \eqref{Mountainpass}, $I_{\la,\delta}(\psi_\theta)\to -\infty$ as $\theta\to\infty$. Also we have $I_{\la,\delta}(\varphi_\theta)>0$ for $\theta>0$ small enough. Consequently, $c_\la\leq M_0<+\infty$, where $M_0$ is independent on $\la$ and $\mu$.
\end{proof}

\section{The proof of Theorem \ref{maintheorem1}}
In this section, we will prove the Theorem \ref{maintheorem1} in detail. Firstly we
we introduce the following Poho\u{z}aev identity (see \cite{Teng2}):

\begin{lemma}\label{Pohozaev}
(Poho\u{z}aev identity) Let $u\in H^s(\R^3)$ be a critical point of the functional $I_{\la,\mu}$ ($\forall \mu\in[\delta,1]$) given by \eqref{IVK}, then we have the following Poho\u{z}aev identity:
\begin{equation}\label{Pohozaev1}
 \begin{gathered}
P_{\la,\mu}(u)\triangleq \frac{3-2s}{2}\int_{\R^3}|(-\Delta)^{\frac{s}{2}}u|^2dx+\frac{3}{2}\int_{\R^3}\la V(x)|u|^2dx
+\frac{1}{2}\int_{\R^3}\la(x,\nabla V)|u|^2dx\hfill\\
+\frac{2t+3}{4}\int_{\R^3}K(x)\phi_u^t u^2 dx+\frac{1}{2}\int_{\R^3}(x,\nabla K)\phi_u^t u^2 dx
-3\mu\int_{\R^3}F(u)dx\equiv0.\hfill\\
 \end{gathered}
\end{equation}
\end{lemma}

\begin{lemma}\label{PScondition}
Let $\{u_n\}$ be a bounded $(PS)$ sequence of the functional $I_{\la,\mu}$ ($\forall \mu\in[\delta,1]$) at the level $c>0$, then for any $M>c$, there exists $\Lambda=\Lambda(M)>0$ such that $\{u_n\}$ contains a strongly convergent subsequence in $E_\la$ for all $\la>\Lambda$.
\end{lemma}

\begin{proof}
Since $\{u_n\}$ is bounded in $E_\la$, then there exists $u\in E_\la$ such that $u_n\rightharpoonup u$ in $E_\la$, $u_n\to u$ in $L^m_{\text{loc}}(\R^3)$ with $m\in[1,2_s^*)$ and $u_n\to u$ $a.e.$ in $\R^3$. To show the proof clearly, we will split it into several steps:
\vskip0.3cm
\underline{\textbf{Step 1:}} \ \ \ \   $I_{\la,\mu}^{\prime}(u)=0$ \ \ and\ \  $I_{\la,\mu}(u)\geq 0$.
\vskip0.3cm
\noindent To show $I_\la^{\prime}(u)=0$, since $C_0^{\infty}(\R^3)$ is dense in $E_\la$, then it suffices to show
\[
\langle I_{\la,\mu}^{\prime}(u),\varphi\rangle=0\ \ \text{for any}\ \ \varphi\in C_0^{\infty}(\R^3).
\]
It is totally similar to the proof of \cite[(3.2)]{Shen2} that
\[
\int_{\R^3}f(u_n)\varphi dx\to \int_{\R^3}f(u)\varphi dx.
\]
Using the above formula and \eqref{weak2}, one has
\[
\langle I_{\la,\mu}^{\prime}(u),\varphi\rangle=\lim_{n\to\infty}\langle I_{\la,\mu}^{\prime}(u_n),\varphi\rangle=0.
\]
Since $u$ is a critical point of $I_{\la,\mu}$, then by \eqref{Pohozaev1} one has
\begin{eqnarray*}
 I_{\la,\mu}(u) &=& I_{\la,\mu}(u)-\frac{1}{(s+t)\gamma-3}\big[(s+t)\langle I_{\la,\mu}^{\prime}(u),u\rangle-P_{\la.\mu}(u)\big]\\
   &=&\frac{(s+t)\gamma-(4s+2t)}{2\big[(s+t)\gamma-3\big]}\int_{\R^3}|(-\Delta)^{\frac{s}{2}}u|^2dx
   +\frac{s+t}{(s+t)\gamma-3}\int_{\R^3}\big[uf(u)-\gamma F(u)\big]dx\\
   &&+\frac{(s+t)(\gamma-2)}{2\big[(s+t)\gamma-3\big]}\int_{\R^3}\la V(x)u^2dx+
   \frac{1}{2\big[(s+t)\gamma-3\big]}\int_{\R^3}\la(x,\nabla V)u^2dx \\
   &&+\frac{(s+t)\gamma-(4s+2t)}{4\big[(s+t)\gamma-3\big]}\int_{\R^3}K(x)\phi_u^t u^2 dx
   +\frac{1}{2\big[(s+t)\gamma-3\big]}\int_{\R^3}(x,\nabla K)\phi_u^t u^2 dx \\
  &\geq&\frac{(s+t)\gamma-(4s+2t)}{2\big[(s+t)\gamma-3\big]}\int_{\R^3}|(-\Delta)^{\frac{s}{2}}u|^2dx\geq0,
 \end{eqnarray*}
 where we have used the fact $\gamma>\frac{4s+2t}{s+t}$ implies that $(s+t)\gamma>3$.
\vskip0.3cm
\underline{\textbf{Step 2:}} \ \ \ \   $u_n\to u$ \ \ in\ \  $E_\la$.
\vskip0.3cm
\noindent Let $v_n\triangleq u_n-u$, by \eqref{weak1}, \eqref{weak2} and the Br\'{e}zis-Lieb lemma \cite{Brezis}, one has
\begin{equation}\label{vn}
  I_{\la,\mu}(v_n)=I_{\la,\mu}(u_n)-I_{\la,\mu}(u)+o(1)\ \ \text{and}\ \ I^{\prime}_{\la,\mu}(v_n)=I^{\prime}_\la(u_n)+o(1).
\end{equation}
As a consequence of the condition $(V_2)$ and the locally compact Sobolev imbedding theorem, one has
\begin{align*}
   \int_{\R^3}v_n^2dx & =\int_{\{V\geq c\}}v_n^2dx+\int_{\{V<c\}}v_n^2dx=\int_{\{V\geq c\}}v_n^2dx+o(1)  \\
  & \leq \frac{1}{\la c} \int_{\{V\geq c\}}\la V(x)v_n^2dx+o(1)\leq \frac{1}{\la c}\|v_n\|_\la^2+o(1)
\end{align*}
which implies that
 \begin{eqnarray}\label{Lq}
\nonumber \int_{\R^3}|v_n|^qdx&=& \bigg(\int_{\R^3}|v_n|^2dx\bigg)^{\frac{2_s^*-q}{2_s^*-2}}\bigg(\int_{\R^3}|v_n|^{2_s^*}dx\bigg)^{\frac{q-2}{2_s^*-2}} \\
\nonumber  &\leq&\bigg(\frac{1}{\la c}\bigg)^{\frac{2_s^*-q}{2_s^*-2}}\|v_n\|_\la^{2{\frac{2_s^*-q}{2_s^*-2}}}
 S_s^{-\frac{2_s^*}{2}{\frac{q-2}{2_s^*-2}}}\|v_n\|_\la^{2_s^*{\frac{q-2}{2_s^*-2}}}+o(1)  \\
   &=&  \bigg(\frac{1}{\la c}\bigg)^{\frac{2_s^*-q}{2_s^*-2}}
 S_s^{-\frac{2_s^*}{2}{\frac{q-2}{2_s^*-2}}}\|v_n\|_\la^q+o(1)
 \end{eqnarray}
Using $I_{\la,\mu}(u)\geq 0$ in Step 1, $(f_3)$ and \eqref{vn}, we derive
\begin{eqnarray}\label{boundness}
\nonumber  M_0 &\geq &c-I_\la(u)\ \ =\ \ I_\la(v_n)-\frac{1}{\gamma}\langle I_\la^{\prime}(v_n),v_n\rangle +o(1)  \\
   &=&\frac{\gamma-2}{2\gamma}\|v_n\|_\la^2+ \frac{\gamma-4}{4\gamma} \nonumber   \int_{\R^3}K(x)\phi_{v_n}^tv_n^2dx+\int_{\R^3}\big[\frac{1}{\gamma}v_nf(v_n)-F(v_n)\big]dx \\
 &\stackrel{\mathrm{\eqref{weak1}}}{\geq}& \frac{\gamma-2}{2\gamma}\|v_n\|_\la^2+ o(1)
 \stackrel{\mathrm{\eqref{Lr}}}{\geq}\frac{\gamma-2}{2\gamma}  \bigg(2\big|\{V<c\}\big|\bigg)^{-\frac{2}{q}\frac{2_s^*-q}{2_s^*}}
S_s |v_n|_q^2+ o(1)
\end{eqnarray}
when $\la\geq c^{-1}\big|\{V<c\}\big|^{-\frac{2_s^*-2}{2_s^*}}S_s$.

Combing with \eqref{Lq} and \eqref{boundness}, for any $\la\geq c^{-1}\big|\{V<c\}\big|^{-\frac{2_s^*-2}{2_s^*}}S_s$, we have
\begin{eqnarray*}
\int_{\R^3}|v_n|^qdx&\leq& \bigg(\int_{\R^3}|v_n|^qdx\bigg)^{\frac{q-2}{q}}
\bigg(\int_{\R^3}|v_n|^qdx\bigg)^{\frac{2}{q}} \\
  &\leq& \bigg(\frac{2M\gamma}{S_s(\gamma-2)}\bigg)^{\frac{q-2}{2}} \bigg(2\big|\{V<c\}\big|\bigg)^{\frac{q-2}{q}\frac{2_s^*-q}{2_s^*}}
  \bigg(\frac{1}{\la c}\bigg)^{\frac{2}{q}\frac{2_s^*-q}{2_s^*-2}}
 S_s^{-\frac{2_s^*}{q}{\frac{q-2}{2_s^*-2}}}\|v_n\|_\la^2
\end{eqnarray*}
which reveals that
\begin{eqnarray*}
 o(1)&=&\langle I_\la^{\prime}(v_n),v_n\rangle \\
  &=&\|v_n\|_\la^2+\int_{\R^3}K(x)\phi_{v_n}^tv_n^2dx-\int_{\R^3}v_nf(v_n)dx   \\
   &{\geq}& \|v_n\|_\la^2- \epsilon\int_{\R^3}v_n^2dx-C_\epsilon\int_{\R^3}|v_n|^qdx  \\
   &\stackrel{\mathrm{\eqref{Lr}}}{\geq}&\Bigg\{1-\epsilon \bigg(2\big|\{V<c\}\big|\bigg)^{\frac{2_s^*-2}{2_s^*}}
S_s^{-\frac{r}{2}}\\
&&-C_\epsilon\bigg(\frac{2M\gamma}{S_s(\gamma-2)}\bigg)^{\frac{q-2}{2}} \bigg(2\big|\{V<c\}\big|\bigg)^{\frac{q-2}{q}\frac{2_s^*-q}{2_s^*}}
  \bigg(\frac{1}{\la c}\bigg)^{\frac{2}{q}\frac{2_s^*-q}{2_s^*-2}}
 S_s^{-\frac{2_s^*}{q}{\frac{q-2}{2_s^*-2}}}\Bigg\}  \|v_n\|_\la^2
\end{eqnarray*}
when $\la\geq c^{-1}\big|\{V<c\}\big|^{-\frac{2_s^*-2}{2_s^*}}S_s$. Therefore if we take $\epsilon>0$ sufficiently small, then there exists $\Lambda=\Lambda(M)>c^{-1}\big|\{V<c\}\big|^{-\frac{2_s^*-2}{2_s^*}}S_s$ such that $\|v_n\|_\la\to 0$ as $n\to\infty$.
\end{proof}

As a direct consequence Proposition \ref{proposition}, Lemma \ref{2Mountpass} and Lemma \ref{PScondition}, there exist two sequences $\{\mu_n\}\subset[\delta,1]$ and $\{u_n\}\subset E_\la \backslash \{0\}$ (we denote $\{u(\mu_n)$ by $\{u_n\}$ just for simplicity) such that
\begin{equation}\label{proof1}
  I^{\prime}_{\la,\mu_n}(u_n)=0,\ \ I_{\la,\mu_n}(u_n)=c_{\la,\mu_n}\ \ \text{and}\ \ \mu_n\to1^-.
\end{equation}

We are now ready to prove Theorem \ref{maintheorem1}.

\begin{proof}[\textbf{Proof of Theorem \ref{maintheorem1}}]
We first claim that the sequence given by \eqref{proof1} is bounded. In fact, recalling $(c)$ of Lemma \ref{2Mountpass}, \eqref{Pohozaev1}, the assumptions $(V_4)$ and $(K)$, one has
 \begin{eqnarray}\label{proof12}
\nonumber M_0  &\geq&c_{\la,\mu_n}= I_{\la,\mu_n}(u_n)-\frac{1}{(s+t)\gamma-3}\big[(s+t)\langle I_{\la,\mu_n}^{\prime}(u_n),u_n\rangle-P_{\la,\mu_n}(u_n)\big]\\
\nonumber   &=&\frac{(s+t)\gamma-(4s+2t)}{2\big[(s+t)\gamma-3\big]}\int_{\R^3}|(-\Delta)^{\frac{s}{2}}u_n|^2dx
   +\frac{s+t}{(s+t)\gamma-3}\int_{\R^3}\big[u_nf(u_n)-\gamma F(u_n)\big]dx\\
 \nonumber  &&+\frac{(s+t)(\gamma-2)}{2\big[(s+t)\gamma-3\big]}\int_{\R^3}\la V(x)u_n^2dx+
   \frac{1}{2\big[(s+t)\gamma-3\big]}\int_{\R^3}\la(x,\nabla V)u_n^2dx \\
 \nonumber  &&+\frac{(s+t)\gamma-(4s+2t)}{4\big[(s+t)\gamma-3\big]}\int_{\R^3}K(x)\phi_{u_n}^t u_n^2 dx
   +\frac{1}{2\big[(s+t)\gamma-3\big]}\int_{\R^3}(x,\nabla K)\phi_{u_n}^t u_n^2 dx \\
  &\geq&\frac{(s+t)\gamma-(4s+2t)}{2\big[(s+t)\gamma-3\big]}\int_{\R^3}|(-\Delta)^{\frac{s}{2}}u_n|^2dx
 \end{eqnarray}
which shows that $|(-\Delta)^{\frac{s}{2}}u_n|_2$ is bounded. By interpolation inequality, for $q\in(2,2_s^*)$ one has
\begin{equation}\label{proof9}
  \begin{gathered}
|u_n|_q^q\leq |u_n|_2^{2\xi}|u_n|_{2_s^*}^{2_s^*(1-\xi)}\stackrel{\mathrm{\eqref{Lr}}}{\leq} C\|u\|_\la^{2\xi}|u_n|_{2_s^*}^{2(1-\xi)}\hfill\\
\ \ \ \ \ \ \stackrel{\mathrm{\eqref{Sobolev2}}}{\leq}C\|u_n\|_\la^{2\xi}S_s^{-(1-\xi)}|(-\Delta)^{\frac{s}{2}}u_n|_2^{1-\xi}
\leq C\|u_n\|_\la^{2\xi},\hfill\\
\end{gathered}
\end{equation}
where $\xi=\frac{2_s^*-q}{2_s^*-2}\in(0,1)$. Therefore by \eqref{growth2}, one has
\begin{align*}
  M_0  &\geq c_{\la,\mu_n} =I_{\la,\mu_n}(u_n)\geq \frac{1}{2}\|u_n\|_{\la}^2-\epsilon\|u_n\|_\la^2-C_\epsilon |u_n|_q^q \\
    &  \geq \frac{1}{4}\|u_n\|_{\la}^2-C\|u_n\|_\la^{2\xi},
\end{align*}
which implies that $\{u_n\}$ is bounded in $E_\la$ because $\xi\in(0,1)$.

Since $\mu_n\to1^-$, we claim that $\{u_n\}$ is a $(PS)_{c_{\la,1}}$ sequence of the functional $I_\la=I_{\la,1}$. In fact, as a consequence of  Lemma \ref{proposition} $(c)$ we obtain that
\[
\lim_{n\to\infty} I_{\la,1}(u_n)=\bigg(\lim_{n\to\infty}I_{\la,\mu_n}(u_n)
+(\mu_n-1)\int_{\R^3}F(u_n)dx\bigg)\stackrel{\mathrm{\eqref{growth1}}}{=}\lim_{n\to\infty}c_{\la,\mu_n}=c_{\la.1}
\]
and for all $\psi\in H^s(\R^3)\backslash\{0\}$,
\begin{eqnarray*}
 \lim_{n\to\infty} \frac{|\langle I^{\prime}_{\la,1}(u_n),\psi\rangle|}{\|\psi\|}&=& \lim_{n\to\infty}\frac{\big|\langle I^{\prime}_{\la,\mu_n}(u_n),\psi\rangle
+ (\mu_n-1) \int_{\R^3}f(u_n)\psi dx\big|}{\|\psi\|} \\
   &=&\lim_{n\to\infty} \frac{|\mu_n-1|\big|\int_{\R^3}f(u_n)\psi dx\big|}{ \|\psi\|} \\
   &\stackrel{\mathrm{\eqref{growth1}}}{\leq}& \lim_{n\to\infty} |\mu_n-1|(\epsilon\|u_n\|+C_\epsilon\|u_n\|^{q-1})  \to 0,
\end{eqnarray*}
which imply that  $\{u_n\}$ is a $(PS)_{c_{\la,1}}$ sequence of $I_\la=I_{\la,1}$ at the level $c_{\la,1}>0$, where we have used the fact that $\{u_n\}$ is bounded in $E$. Consequently by Lemma \ref{PScondition}, there exists a subsequence still denoted by itself such that $u_n\to u$ in $E$ which implies that $I_{\la}(u)=c_{\la,1}>0$ and $I^{\prime}_{\la}(u)=0$.

Inspired by J. Sun and S. Ma \cite{Sun}, to obtain a ground state solution we set
\[
m=\inf\big\{I_\la(u):u\in E\backslash\{0\},\ \ I_\la^{\prime}(u)=0\big\}.
\]
We claim that $m>0$. Indeed, similar to the Step 1 in the proof of Lemma \ref{PScondition}, one has $m\geq 0$. In order to show $m>0$, we suppose that $m=0$. Take a minimizing sequence $\{w_n\}$ such that $I^{\prime}_\la(w_n)=0$ and $I_\la(w_n)\to 0$. Using $I^{\prime}_\la(w_n)=0$ and \eqref{growth1}, one has
\begin{equation}\label{proof10}
  \|w_n\|^2\leq\|w_n\|_\la^2\leq\int_{\R^3}f(w_n)w_n \leq  \frac{1}{2}\|w_n\|^2+C|w_n|_q^q \stackrel{\mathrm{\eqref{Sobolev3}}}{\leq} \frac{1}{2}\|w_n\|^2+C\|w_n\|^q
\end{equation}
which implies that $\|w_n\|\geq C>0$ for some $C$ independent of $n$. On the other hand, Using $I_\la(w_n)\to 0$ and $I^{\prime}_\la(w_n)=0$, as \eqref{proof12} we have $|(-\Delta)^{\frac{s}{2}}w_n|_2\to 0$. Similar to the Step 1 in the proof of Lemma \ref{PScondition}, $\{w_n\}$ is bounded in $E_\la$. Hence $|w_n|_q\to 0$ by \eqref{proof9}. Using \eqref{proof10} again, we have $0\leq \|w_n\|^2\leq C|w_n|_q^q \to 0$, which is a contradiction!

Suppose that there exists a sequence $\{u_n\}\subset E\backslash\{0\}$ such that $I_\la^{\prime}(u_n)=0$ and $I_\la(u_n)\to m$. We can conclude that $\{u_n\}$ is bounded in $E$, and then $\{u_n\}$ is $(PS)$ sequence at the level $m>0$. By Lemma \ref{PScondition}, passing to a subsequence if necessary, $u_n\to u$ in $E_\la$. Hence we have that $I_\la(u)= m>0$ and $I^{\prime}_\la(u)=0$ which shows that $u$ is a nontrivial critical point of $I_\la$ given by \eqref{functional}. It follows from \cite[Proposition 4.4]{Teng3} that $u$ is positive. Therefore $(u,\phi_u)$ is a positive ground state to system \eqref{mainequation1}. The proof is complete.
 \end{proof}

\section{Concentration for the nontrivial solutions obtained in Theorem \ref{maintheorem1}}
Before we study the concentration results, let us recall the Vanishing lemma for fractional Sobolev space as follows.
\begin{lemma}\label{Vanishing}
\
\big(see \cite[Lemma 2.4]{Secchi}\big) Assume that $\{u_n\}$ is bounded in $H^\alpha(\R^3)$ for $\alpha\in(0,1)$ and satisfies
\[
\lim_{n\to\infty}\sup_{y\in\R^3}\int_{B_\rho(y)}|u_n|^2dx=0,
\]
for some $\rho>0$. Then $u_n\to 0$ in $L^m(\R^3)$ for every $2<m<2_\alpha^*$.
\end{lemma}

We adapt the idea used in \cite{Bartsch,Zhao} to prove Theorem \ref{maintheorem2}.

\begin{proof}[\textbf{Proof of Theorem \ref{maintheorem2}}]
For any sequence $\la_n\to\infty$, we denote $\{u_n\}$ to be the positive ground state solutions $\{u_{\la_n}\}$ obtained in Theorem \ref{maintheorem1}. It is similar to the proof in Theorem \ref{maintheorem1} that $\{u_n\}$ is bounded in $H^s(\R^3)$ and going to a subsequence if necessary we can assume that $u_n\rightharpoonup u_0$ in $E$, $u_n\to u_0$ in $L^p_{\text{loc}}(\R^3)$ with $p\in[1,2_s^*)$ and $u_n\to u_0$ in $a.e.$ in $\R^3$. Using Fatou's lemma, one has
\[
0\leq \int_{\R^3\backslash V^{-1}(0)}V(x)u_0^2dx\leq \mathop{\lim\inf}_{n\to\infty}\int_{\R^3}V(x)u_n^2dx\leq \mathop{\lim\inf}_{n\to\infty}\frac{1}{\la_n}\|u_n\|_{\la_n}^2\leq \lim_{n\to\infty}\frac{C}{\la_n}=0
\]
which implies that $u_0=0$ $a.e.$ in $\R^3\backslash V^{-1}(0)$, then we have that $u_0\in H_0^s(\Omega)$ because $\Omega=int V^{-1}(0)$ by $(V_3)$. Now for any $\varphi\in C_0^{\infty}(\Omega)$, and since $\langle I^{\prime}_{\la_n}(u_n),\varphi\rangle=0$, we can easily check that
\[
\int_{\R^3}(-\Delta)^{\frac{s}{2}}u_0(-\Delta)^{\frac{s}{2}}\varphi+\int_{\R^3}K(x)\phi_{u_0}^tu_0\varphi dx-\int_{\R^3}f(u_0)\varphi dx=0.
\]
As $C_0^{\infty}(\Omega)$ is dense in $H_0^s(\Omega)$, $u_0$ is a solution of \eqref{mainequation2}.

We claim that $u_n\to u_0$ in $L^q(\R^3)$ for $q\in(2,2_s^*)$. Arguing it by indirectly, then by Lemma \ref{Vanishing} there exists $\{y_n\}\subset\R^3$, $\rho>0$ and $\delta_0>0$ such that
\[
\int_{B_\rho(y_n)}(u_n-u_0)^2dx\geq \delta_0>0,
\]
where $|y_n|\to\infty$ which implies that $\big|B_\rho(y_n)\cap\{V<c\}\big|\to 0$. By H\"{o}lder's inequality
\[
\int_{B_\rho(y_n)\cap\{V<c\}}(u_n-u_0)^2dx\to 0
\]
which implies that for sufficiently large $n$ one has
\[
\int_{B_\rho(y_n)\cap\{V<c\}}(u_n-u_0)^2dx\leq \frac{\delta_0}{2}.
\]
Therefore for sufficiently large $n$ and $u_0=0$ $a.e.$ in $\R^3\backslash V^{-1}(0)$ give that
\begin{align*}
 \|u_n\|_{\la_n}^2& \geq \la_n\int_{\R^3}V(x)u_n^2dx \geq \la_nc\int_{B_\rho(y_n)\cap \{v\geq c\}}u_n^2dx
= \la_nc\int_{B_\rho(y_n)\cap \{v\geq c\}}(u_n-u_0)^2dx \\
    &=\la_nc\bigg(\int_{B_\rho(y_n)}(u_n-u_0)^2dx-\int_{B_\rho(y_n)\cap \{v<c\}}(u_n-u_0)^2dx\bigg)\geq \frac{\delta_0}{2}\la_nc\to \infty
\end{align*}
which yields a contradiction! Hence $u_n\to u_0$ in $L^q(\R^3)$ for $q\in(2,2_s^*)$ which implies that
\[
\int_{\R^3}f(u_n)u_ndx-\int_{\R^3}f(u_n)u_0dx=o(1)
\]
by the Strass compactness theorem in \cite{Berestycki}.

We now show that $u_n\to u_0$ in $E$. In fact, by $\langle I_{\la_n}^{\prime}u_n,u_n\rangle=\langle I_{\la_n}^{\prime}u_n,u_0\rangle=0$
\begin{equation}\label{proof8}
 \|u_n\|_{\la_n}^2+\int_{\R^3}K(x)\phi_{u_n}^tu_n^2dx=\int_{\R^3}f(u_n)u_ndx
\end{equation}
and
\[
(u_n,u_0)_{\la_n}+\int_{\R^3}K(x)\phi_{u_n}^tu_nu_0dx=\int_{\R^3}f(u_n)u_0dx.
\]
In view of the definition $A_2$ in the proof of \eqref{weak1}, one has
\[
\int_{\R^3}K(x)\phi_{u_n}^tu_n^2dx-\int_{\R^3}K(x)\phi_{u_n}^tu_nu_0dx=o(1).
\]
Hence by the above four formulas we have that
\[
\lim_{n\to\infty}\|u_n\|_{\la_n}^2=\lim_{n\to\infty}(u_n,u_0)_{\la_n}=\lim_{n\to\infty}(u_n,u_0)=\|u_0\|^2.
\]
Also since the norm is lower semicontinuous, then
\[
\|u_0\|^2\leq\mathop{\inf\lim}_{n\to\infty}\|u_n\|^2\leq \mathop{\inf\lim}_{n\to\infty}\|u_n\|_{\la_n}^2,
\]
and thus $u_n\to u_0$ in $E$.

Finally, we show $u_0\not\equiv 0$. Using \eqref{Sobolev2}, \eqref{growth1} and \eqref{proof8}, we drive
\[
\|u_n\|^2\leq \|u_n\|^2_{\la_n}\leq\int_{\R^3}f(u_n)u_ndx
\leq \frac{1}{2}\|u_n\|^2+C\|u_n\|_q^q
\]
which implies $\|u_0\|^2\leq C\|u_0\|^q$ together with $u_n\to u_0$ in $E$. Therefore $u_0\not\equiv 0$ for $q\in(2,2_s^*)$.
The proof is complete.
\end{proof}

\end{document}